\documentclass[]{interact}

\usepackage[caption=false]{subfig}% Support for small, `sub' figures and tables
\usepackage{multirow,url}
\usepackage[round,authoryear]{natbib}
\allowdisplaybreaks
\def\mR{\mathbb{R}}

\def\E{\mathbb{E}}

\def\tr{\mbox{tr}}
\def\var{\mbox{var}}

\newcommand{\bSig}{\mbox{\boldmath $\Sigma$}}

\newcommand{\bmu}{\mbox{\boldmath $\mu$}}

\def\U{{\bm{U}}}
\def\X{\bm{X}}
\def\x{\bm{x}}
\def\Y{\bm{Y}}

\def\Z{\bm {Z}}

\def\A{\bm {A}}

\def\B{\bm{ B}}

\def\C{\bm {C}}

\def\bI{\bm{I}}

\def\Sig{\bm{\Sigma}}
\newcommand{\trans}{^\top}
\def\defby{\stackrel{\mbox{\textrm{\tiny def}}}{=}}
\def\topr{\,{\buildrel p \over \longrightarrow}\,}
\def\tod{\,{\buildrel d \over \longrightarrow}\,}

\theoremstyle{plain}% Theorem-like structures provided by amsthm.sty
\newtheorem{theorem}{Theorem}[section]
\newtheorem{lemma}[theorem]{Lemma}

\newtheorem{assum}{Assumption}[section]

\theoremstyle{definition}

\newtheorem{example}{Example}[section]

\theoremstyle{remark}
\newtheorem{remark}{Remark}

\begin{document}

%\articletype{ARTICLE TEMPLATE}% Specify the article type or omit as appropriate

\title{Statistical inference on kurtosis of elliptical distributions}

\author{
\name{Bowen Zhou, Peirong Xu and Cheng Wang\thanks{CONTACT Cheng Wang. Email: chengwang@sjtu.edu.cn}}
\affil{School of Mathematical Science, Shanghai Jiao Tong University, Shanghai, 20040,  China}
}

\maketitle

\begin{abstract}
Multivariate elliptically-contoured distributions are widely used for modeling correlated and non-Gaussian data. In this work, we study the kurtosis of the elliptical model, which is an important parameter in many statistical analysis. Based on U-statistics, we develop an estimation method. Theoretically, we show that the proposed estimator is consistent under regular conditions, especially we relax a moment condition and the restriction that the data dimension and the sample size are of the same order. Furthermore, we derive the asymptotic normality of the estimator and evaluate the asymptotic variance through several examples, which allows us to construct a confidence interval. The performance of our method is validated by extensive simulations and real data analysis. 
\end{abstract}


\section{Introduction}
The multivariate normal distribution plays a central role in multivariate statistical analysis but is often violated in practical applications. For instance, in finance, while stock returns are symmetric, they exhibit leptokurtosis\citep{vcivzek2011statistical, mcneil2015quantitative}. In genomics and bioimaging, empirical evidence suggests that the Gaussian assumption may not hold\citep{thomas2010validation, posekany2011biological}. As a natural extension, elliptical distributions have received considerable attention in the past few decades \citep{fang1990generalized,gupta2013elliptically}. This class of distributions retains many desirable properties of the multivariate normal distribution, such as symmetry, while also enabling the modeling of non-normal dependence and multivariate extremes, making them valuable in various applications. In particular, many methods in high-dimensional data analysis have been motivated by the study of elliptical distributions \citep{han2012transelliptical, fan2015quadro, fan2018large}.

A random vector $\X \in \mR^p$ follows an elliptical distribution if it can be represented in the form
\begin{align}\label{eq:model}
	\X=\bmu+\xi\bSig^{\frac{1}{2}}\U,
\end{align}
where $\bmu \in \mR^p$ is a $p$-dimensional constant vector, $\bSig \in \mR^{p \times p}$ is a $p\times p$ positive definite matrix, $\U \in \mR^{p}$ is a random vector uniformly distributed on the unit sphere $\mathcal{S}^{p-1}$, and $\xi \in [0, \infty)$ is a random variable that independent of $\U$. Given different distributions of $\xi$, the class of elliptical distributions encompasses several important distribution families. For instance, it is reduced to the multivariate normal distribution when $\xi^2$ follows a chi-square distribution. If $\xi^2$ follows an F-distribution, it results in the multivariate $t$ distribution. In general, selecting an appropriate distribution for $\xi$ is crucial but challenging for real data analysis, often leveraging the moment information of $\xi$. In particular, the kurtosis parameter \citep{ke2018higher}
\begin{align*}
	\theta \defby \frac{\E\xi^4}{p(p+2)}
\end{align*}
is widely discussed in the literature. For example, \cite{fan2015quadro} demonstrated that an estimator of 
$\theta$ is essential for constructing a general quadratic classifier under the assumption that data are generated from elliptical distributions. In the context of financial assets, the leptokurtosis, which equals $\theta-1$, is often used to capture the tail behavior of stock returns \citep{ke2018higher}. In random matrix theory, \cite{hu2019high} utilized the established asymptotic properties for the spherical test of the covariance matrix, which requires an estimator of 
$\theta$ to construct the test statistic.  \cite{Wang2023boot} further proposed the parametric bootstrapping method for inference with spectral statistics in high-dimensional elliptical models, requiring a consistent estimator of $\var(\xi^2/p)=(p+2)\theta/p-1$. Therefore, constructing a theoretically justified estimator for $\theta$ is crucial for facilitating statistical inference in different high-dimensional settings.

For the classical setting where the data dimension of $\X$ is much smaller than the sample size, \cite{seo1996estimation} proposed a moment estimator by noting $\xi^2=(\X-\bmu)\trans\bSig^{-1}(\X-\bmu)$, where $\Sigma$ can be well estimated by the sample covariance matrix. It performs poorly in the high-dimensional setting due to noise accumulation of the sample covariance matrix. To address this problem, \cite{fan2015quadro} proposed a moment estimator using the shrinkage estimator of the precision matrix $\Sigma^{-1}$. The consistency of this estimator requires a sparsity condition on $\Sigma^{-1}$ and is restricted to the sub-Gaussian family. \cite{ke2018higher} proposed an estimator of $\theta$ based on the coordinates of the data, avoiding the estimation of $\Sigma^{-1}$ but neglecting the within-coordinate correlation. The estimator is consistent in ultra-high dimensional cases, but its asymptotic normality holds only when the data are from a sub-Gaussian distribution. \cite{Wang2023boot} studied the estimation of $\var(\xi^2/p)=(p+2)\theta/p-1$ directly by constructing a quadratic-form estimating equation. The estimator is consistent when $p$ is of the same order as the sample size $n$, but its asymptotic normality remains unknown.

We propose to estimate the kurtosis by representing it as a function of U-statistics. Our work contributes to the existing literature in several theoretical aspects: (1) both the consistency and the asymptotic normality of the estimator hold for data collected from general elliptical distributions; (2) the estimator achieves a faster convergence rate than that in \cite{ke2018higher} even when $p$ is larger than $n$; and (3) the asymptotic normality holds for both light-tailed and heavy-tailed cases, enabling the construction of valid confidence intervals for different distribution families.

The rest of the paper is organized as follows. We propose an model-free estimator of $\theta$ for high-dimensional settings in Section \ref{sec2-1}, and study its theoretical properties in Section \ref{sec2-2}. In Section \ref{sec2:exm}, we discuss the asymptotic normality of the estimator and construct confidence intervals of $\theta$ under several commonly used distribution families. We investigate the finite sample performance of our method with other competing methods in Section \ref{sec3}. In Section \ref{sec4},  we demonstrate the effectiveness of the proposed method through several real data applications. All technical proofs are provided in the Appendix \ref{sec:a}.

\section{Main results} \label{sec2}

\subsection{Methodology}\label{sec2-1}
Suppose $\X \in \mR^p$ follows an elliptical distribution (\ref{eq:model}), where assuming $\E\xi^2=p$ for model identifiability. Then, we know that
\begin{align*}
	\var(\|\X-\bmu\|_2^2)=(\theta-1)\tr^2\bSig+2\theta\tr\bSig^2,
\end{align*}
which induces a moment equation
\begin{align}
	\theta=\frac{\var(\|\X-\bmu\|_2^2)+\tr^2\bSig}{\tr^2\bSig+2\tr\bSig^2}.\label{eq:mom}
\end{align}
This implies that an empirical estimate of $\theta$ can be derived via plug-in technique by estimating $\var(\|\X-\bmu\|_2^2)$, $\tr^2\bSig$ and $\tr\bSig^2$, respectively.

From the perspective of a $U$-statistic, we have
\begin{align*}
	\var(\|\X-\bmu\|_2^2)=&\frac{1}{4}\E\left(\|\X_1-\X_2\|_2^2-\|\X_3-\X_4\|_2^2\right)^2-2\tr\bSig^2,\\
	\tr^2\bSig=&\frac{1}{4}\E\|\X_1-\X_2\|_2^2\|\X_3-\X_4\|_2^2,\\
	\tr\bSig^2=&\frac{1}{4}\E\left((\X_1-\X_2)\trans(\X_3-\X_4)\right)^2,
\end{align*}
where $\X_1, \X_2\cdots, \X_4$ are i.i.d. copies of $\X$. Therefore, based on i.i.d. samples $\X_1,\ldots,\X_n \in \mR^p$ from the elliptical distribution \eqref{eq:model}, empirical estimates of $\var(\|\X-\bmu\|_2^2)$, $\tr^2\bSig$ and $\tr\bSig^2$ would be $T_1-2T_3, T_2, T_3$, respectively, where
\begin{align*}
	&T_1=\frac{1}{4n(n-1)(n-2)(n-3)}\sum_{i\neq j\neq k\neq l}\left(\|\X_i-\X_j\|_2^2-\|\X_k-\X_l\|_2^2\right)^2,\\
	&T_2=\frac{1}{4n(n-1)(n-2)(n-3)}\sum_{i\neq j\neq k\neq l}\|\X_i-\X_j\|_2^2\|\X_k-\X_l\|_2^2,\\
	&T_3=\frac{1}{4n(n-1)(n-2)(n-3)}\sum_{i\neq j\neq k\neq l}\left((\X_i-\X_j)\trans(\X_k-\X_l)\right)^2.
\end{align*}
The estimator of $\theta$ can be constructed as
\begin{align*}
	\hat\theta_{n}=\frac{T_1+T_2-2T_3}{T_2+2T_3}.
\end{align*}
%\red{The dominant part should be $T_2+2T_3$, not $T_2+T3$. Check all the simulation and proofs. }

\subsection{Consistency}\label{sec2-2}

In this subsection, we establish the consistency and asymptotic normality of the proposed estimator $\theta_n$. We first make the following assumptions:
\begin{assum}\label{assum1}
	As $n\to\infty$, $p=p(n)\to\infty$,  
  $\tr\bSig^4\to \infty$ and   $\tr\bSig^4/\tr^2\bSig^2\to 0$.
\end{assum}
\begin{assum}\label{assum2}
	$\E\xi^8=O(p^4)$.	
\end{assum}
Assumption \ref{assum1} is commonly used in the high-dimensional setting\citep{chen2010tests,chen2010two,guo2016tests}, which is required to control the high-order terms in applying Hoeffding decomposition for U-statistics. If all the eigenvalues of $\bSig$ are bounded, Assumption \ref{assum1} holds naturally. Assumption \ref{assum2} restricts the moments of $\xi^2/p$ up to 4 at the constant scale. It is weaker than that in \cite{Wang2023boot} involving $4+\varepsilon$ moment of $\xi^2/p$ with some $\varepsilon>0$.

% With these conditions, we can show the consistency of the proposed estimator.
\begin{theorem}\label{thm1}
	Under Assumptions \ref{assum1} and \ref{assum2}, we have
	% \begin{align}
	% 	\frac{T_1}{\tr^2\bSig}-(\theta-1)\topr0,~\frac{T_2}{\tr^2\bSig}\topr 1,~\frac{T_3}{\tr^2\bSig}\topr0,\label{eq:topr}
	% \end{align}
	% and thus
	\begin{align*}
		\hat\theta_{n}-\theta\topr0.
	\end{align*}
\end{theorem}

Theorem \ref{thm1} indicates that $\hat\theta_{n}$ is consistent when the data are collected from general elliptical distributions. This extends the consistency result in \cite{ke2018higher} and \cite{Wang2023boot} under weaker conditions.

\begin{theorem}\label{thm2}
Under Assumptions \ref{assum1} and \ref{assum2}, we have that:
\begin{itemize}
    \item[(i)] if $\var(\xi^2/p)=\tau/p+o(1/p)$ and $\var\left((\xi^2/p-1)^2\right)=2\tau^2/p^2+o(1/p^2)$ for some constants $\tau>0$, then
    \begin{align*}
        \frac{\sqrt{n}}{\sigma}\left(\hat\theta_{n}-\theta\right)\tod N(0,1),
    \end{align*}
    where
    \begin{align*}
        \sigma^2=2\left(\frac{\tau-2}{p}+2\frac{\tr\bSig^2}{\tr^2\bSig}\right)^2;
    \end{align*}
    \item[(ii)] if $\var(\xi^2/p)=O(1)$, then
    \begin{align*}
        \frac{\sqrt{n}}{\sigma}\left(\hat\theta_{n}-\theta\right)\tod N(0,1),
    \end{align*}
    where
    \begin{align*}
        \sigma^2=&\var\left[\left(\frac{\xi^2}{p}-1\right)^2-2\var\left(\frac{\xi^2}{p}\right)\cdot\frac{\xi^2}{p}\right].
    \end{align*}
\end{itemize}
\end{theorem}
\begin{remark}\label{rem1}
Case (i) requires $\var(\xi^2/p)$ to be of order $1/p$, similar to the requirements in \cite{hu2019high} and \cite{Wang2023boot}. Additionally, $\var\left((\xi^2/p-1)^2\right)$ must be of order $1/p^2$, which is necessary for establishing the asymptotic normality. These two conditions are satisfied in several widely studied sub-classes of elliptical models. For example, if $\X\sim N(\bmu,\bSig)$, then $\xi\sim \chi^2(p)$ and thus we have
\begin{align*}
     \var\left(\frac{\xi^2}{p}\right)=\frac{2}{p},~\var\left[\left(\frac{\xi^2}{p}-1\right)^2\right]=\frac{8}{p^2}+o\left(\frac{1}{p^2}\right),
 \end{align*}
which implies that the conditions hold with $\tau=2$. More examples can be found in the following section.
\end{remark}
\subsection{Examples}\label{sec2:exm}

In this subsection, we apply the theoretical results in Section \ref{sec2-2}  to derive the confidence interval of the kurtosis parameter $\theta$ for some widely used elliptical distributions.

\begin{example}[Example 2.3 of \citealt{hu2019high}]\label{exm1}
	Let $\X=\bmu+\xi\bSig^{1/2}\U$ with $\xi^2=\sum_{j=1}^{p}Y_j^2$ independent of $\U$, where $\{Y_j,j=1,\cdots,p\}$ is a sequence of i.i.d. random variables with 
	\begin{align*}
		\E Y_j=0,~\E Y_j^2=1,~\E Y_j^4=3+\Delta,~\E Y_j^8<+\infty.
	\end{align*}
The normal distribution is a special case with $Y_j\sim N(0,1),~ j=1,\cdots,p$.	By simple calculation, we have 
 \begin{align*}
     \var\left(\frac{\xi^2}{p}\right)=\frac{2+\Delta}{p},~\var\left[\left(\frac{\xi^2}{p}-1\right)^2\right]=\frac{(2+\Delta)^2}{p^2},
 \end{align*}
which implies that the conditions of Case (i) in Theorem \ref{thm2} holds with $\tau=2+\Delta$. Therefore, 
	\begin{align*}
		\sqrt{\frac{n}{2}}\left(\frac{\Delta}{p}+2\frac{\tr\bSig^2}{\tr^2\bSig}\right)^{-1}(\hat\theta_{n}-\theta)\tod N(0,1).
	\end{align*}
We further estimate $\tr\bSig^2/\tr^2\bSig$ using its consistent estimator $T_3/T_2$ based on Theorem \ref{thm1}, and estimate $\Delta$ by the plug-in estimator $\hat\Delta_n=(p+2)(\hat\theta_n-1)$, leveraging the fact that $\Delta=(p+2)(\theta-1)$. Then, by Slutsky's theorem, we have
 	\begin{align}
		\sqrt{\frac{n}{2}}\left(\frac{\hat\Delta_n}{p}+2\frac{T_3}{T_2}\right)^{-1}(\hat\theta_{n}-\theta)\tod N(0,1).\label{eq:an1}
	\end{align}
Therefore, for a pre-specified nominal level $\alpha \in (0,1)$, the corresponding $1-\alpha$ confidence interval for $\theta$ can be constructed as 
\begin{align}
    \left[\hat\theta_n-\sqrt{\frac{2}{n}}\left(\frac{\hat\Delta_n}{p}+2\frac{T_3}{T_2}\right)Z_{\alpha/2},~\hat\theta_n+\sqrt{\frac{2}{n}}\left(\frac{\hat\Delta_n}{p}+2\frac{T_3}{T_2}\right)Z_{\alpha/2}\right],\label{eq:ci1}
\end{align}
 where $Z_{\alpha/2}$ is the upper $\alpha/2$ quantile of the standard normal distribution.
\end{example}

\begin{example}[Kotz-type distribution of \citealt{kotz1975multivariate}]\label{exm3}
Let $\X=\bmu+\xi\bSig^{1/2}\U$ with $\xi^{2s}\sim $Gamma$(\vartheta,\beta)$, then the $p$-dimensional random vector $\X$ follows the Kotz-type distribution, with the density function
	\begin{align*}
		f(\x)=\frac{s\beta^{\vartheta}\Gamma(p/2)}{\pi^{p/2}\Gamma(\vartheta)}\left((\x-\bmu)\trans\bSig^{-1}(\x-\bmu)\right)^{k-1}\exp\left\{-\beta\left((\x-\bmu)\trans\bSig^{-1}(\x-\bmu)\right)^s\right\},
	\end{align*}
	where $\vartheta=(k-1+p/2)/s>0$ and $\beta,s>0$. 
To ensure the identifiability, normalize $\xi$ to satisfy $\E\xi^2=p$. Then, we can obtain the moments of $\xi^2$ as 
 	\begin{align*}		\E\xi^{2m}=p^m\frac{\Gamma(\vartheta+m/s)\Gamma^{m-1}(\vartheta)}{\Gamma^m(\vartheta+1/s)},~m=1,2,\cdots.
	\end{align*}
 Taking $\vartheta=p$, $\beta=1$ and $s=1/2$ as a special case, we have 
 \begin{align*}
		\E\xi^{2m}=\frac{\prod_{j=1}^{2m}(p+j-1)}{(p+1)^m}, \ m=1,2,\cdots 
	\end{align*}
 which gives
 \begin{align*}
     \var\left(\frac{\xi^2}{p}\right)=\frac{4}{p}+o\left(\frac{1}{p}\right),~\var\left[\left(\frac{\xi^2}{p}-1\right)^2\right]=\frac{32}{p^2}+o\left(\frac{1}{p^2}\right).
 \end{align*}
This implies the Kotz-type distribution satisfies the conditions of Case (i) in Theorem \ref{thm2} with $\tau=4$. Therefore, 
\begin{align*}
		\sqrt{\frac{n}{8}}\left(\frac{1}{p}+\frac{\tr\bSig^2}{\tr^2\bSig}\right)^{-1}\left(\hat\theta_{n}-\theta\right)\tod N(0,1).
	\end{align*}
 Consequently, for a pre-specified nominal level $\alpha \in (0,1)$, the corresponding $1-\alpha$ confidence interval for $\theta$ can be constructed as 
 \begin{align}
		\left[\hat\theta_{n}-\sqrt{\frac{8}{n}}\left(\frac{1}{p}+\frac{T_3}{T_2}\right)Z_{\alpha/2},~\hat\theta_{n}+\sqrt{\frac{8}{n}}\left(\frac{1}{p}+\frac{T_3}{T_2}\right)Z_{\alpha/2}\right].\label{eq:ci3}
	\end{align}	
\end{example}
\begin{example}[Multivariate $t$ distribution]\label{exm2}
 Let $\X=\bmu+\xi\bSig^{1/2}\U$ with $\xi^2\sim p((d-2)/d)F(p,d)$, then the $p$-dimensional random vector $\X$ follows the multivariate $t$ distribution $t_{d}(\bmu,\bSig)$, with the density function
	\begin{align*}
		f(\x)=	\frac{\Gamma[(d+p)/2]}{\Gamma(d/2)d^{p/2}\pi^{p/2}|\bSig|^{1/2}}\left[1+\frac{1}{d}(\x-\bmu)\trans\bSig^{-1}(\x-\bmu)\right]^{-(d+p)/2}.
	\end{align*}
Note that 
 \begin{align*}
     \var\left(\frac{\xi^2}{p}\right)=\frac{2}{d-4}+O\left(\frac{1}{p}\right),
 \end{align*}
 which implies the multivariate $t$ distribution of fixed $d$ satisfies the conditions of Case (ii) in Theorem \ref{thm2}. By simple calculation, we have
 \begin{equation*}
 \sigma^2 = \frac{8(d-2)^2(d+4)}{(d-4)^3(d-6)(d-8)}+O\left(\frac{1}{p}\right)
 \end{equation*}
 which implies that   
	\begin{align*}
		\sqrt{\frac{n(d-4)^3(d-6)(d-8)}{8(d-2)^2(d+4)}}\left(\hat\theta_{n}-\theta\right)\tod N(0,1).
	\end{align*}
We estimate the degree of freedom $d$ by its plug-in estimate 
	\begin{align*}
		d_n=\frac{4\hat\theta_{n}-2}{\hat\theta_{n}-1},
	\end{align*}
	and hence
the $1-\alpha$ confidence interval for $\theta$ is
	\begin{align}
		\left[\hat\theta_{n}-\sqrt{\frac{8(d_n-2)^2(d_n+4)}{n(d_n-4)^3(d_n-6)(d_n-8)}}Z_{\alpha/2}, \hat\theta_{n}+\sqrt{\frac{8(d_n-2)^2(d_n+4)}{n(d_n-4)^3(d_n-6)(d_n-8)}}Z_{\alpha/2}\right].\label{eq:ci2}
	\end{align}
\end{example}

\begin{example}[Multivariate Laplace distribution]\label{exm4}

	The random vector $\X$ follows a normal scale mixture distribution if $\X=\bmu+\xi\bSig^{1/2}\U$ with $\xi=\sqrt{R_1R_2}$, where the random variables $R_1$ and $R_2$ are independent and $R_2\sim\chi^2_p$. In particular, if $R_1$ follows an exponential distribution $\lambda\exp(\lambda)$ with some $\lambda>0$, it reduces to the multivariate Laplace distribution \citep{eltoft2006multivariate}. This distribution satisfies the conditions of Case (ii) in Theorem \ref{thm2} with 
 \begin{align*}
    \var\left(\frac{\xi^2}{p}\right)=1+\frac{4}{p}.
 \end{align*}
 Note that 
\begin{align*}
    \sigma^2=&\var\left[\left(\frac{\xi^2}{p}-1\right)^2-2\var\left(\frac{\xi^2}{p}\right)\cdot\frac{\xi^2}{p}\right]=4+O\left(\frac{1}{p}\right).
\end{align*} 
Therefore,
	\begin{align*}
		\sqrt{n}(\hat\theta_{n}-\theta)\tod N(0,4),%\label{eq:an4}
	\end{align*}
which gives the $1-\alpha$ confidence interval for $\theta$ as
	\begin{align}
		\left[\hat\theta_{n}-\sqrt{\frac{2}{n}}Z_{\alpha/2},~\hat\theta_{n}+\sqrt{\frac{2}{n}}Z_{\alpha/2}\right].\label{eq:ci4}
	\end{align}
\end{example}

\begin{example}[General elliptical distribution]\label{exm5}

    In practice, we typically assume $\X = \bmu+\xi\bSig^{1/2}\U$ follows an elliptical distribution with the distribution of $\xi$ unknown. To derive the confidence interval for $\theta$, we first need to estimate $\sigma^2$ described in Theorem \ref{thm2} properly.
    
\begin{itemize}
    \item If Case (i) holds, noting that $\tau=\lim_{p\to\infty}[(p+2)\theta-p]$, we can estimate $\tau$ by $\hat\tau_n=(p+2)\theta_n-p$, and then give an estimator of $\sigma^2$ by
\begin{align}
    \hat\sigma_n^2=2\left(\frac{\hat\tau_n-2}{p}+2\frac{T_3}{T_2}\right)^2. \label{eq:sign1}
\end{align}
\item If Case (ii) holds, note that
\begin{align*}
  \sigma^2=\frac{\E\xi^8}{p^4}-\left(\frac{\E\xi^4}{p^2}\right)^2-4\frac{\E\xi^6\E\xi^4}{p^5}+4\left(\frac{\E\xi^4}{p^2}\right)^3,
\end{align*}
where
\begin{eqnarray*}
 \E\xi^4 &=& p(p+2)\theta,\\
 \E\xi^6 &=& \frac{p(p+2)(p+4)\E((\X-\bmu)\trans(\X-\bmu))^3}{\tr^3\bSig+6\tr\bSig\tr\bSig^2+8\tr\bSig^3},\\
 \E\xi^8 &=& \frac{p(p+2)(p+4)(p+6)\E((\X-\bmu)\trans(\X-\bmu))^4}{\tr^4\bSig+12\tr^2\bSig\tr\bSig^2+12\tr^2\bSig^2+32\tr\bSig\tr\bSig^3+48\tr\bSig^4}.
\end{eqnarray*}
Let $\bar{\X}$ be the sample mean and $\widehat{\bSig}$ be the sample covariance matrix. Then, we can estimate $\sigma^2$ by
\begin{align}
    \hat\sigma_n^2=\frac{\hat\varphi_n}{p^4}-\left(\frac{(p+2)\hat\theta_n}{p}\right)^2-4\frac{\hat\varrho_n}{p^3}\left(\frac{(p+2)\hat\theta_n}{p}\right)+4\left(\frac{(p+2)\hat\theta_n}{p}\right)^3,\label{eq:sign2}
\end{align}
where
\begin{eqnarray*}
    {\hat\varrho_n} &=& \frac{p(p+2)(p+4)\sum_{i=1}^{n}((\X_i-\bar{\X})\trans(\X_i-\bar{\X}))^3/n}{\tr^3\widehat{\bSig}+6\tr\widehat{\bSig}\tr\widehat{\bSig}^2+8\tr\widehat{\bSig}^3},\\
     {\hat\varphi_n} &=& \frac{p(p+2)(p+4)(p+6)\sum_{i=1}^{n}((\X_i-\bar{\X})\trans(\X_i-\bar{\X}))^4/n}{\tr^4\widehat{\bSig}+12\tr^2\widehat{\bSig}\tr\widehat{\bSig}^2+12\tr^2\widehat{\bSig}^2+32\tr\widehat{\bSig}\tr\widehat{\bSig}^3+48\tr\widehat{\bSig}^4}.
\end{eqnarray*}
\end{itemize}
Given the estimator $\sigma_n^2$, the $1-\alpha$ confidence interval for $\theta$ can be constructed as
\begin{align}
     \left[\hat\theta_n-\frac{\hat\sigma_n}{\sqrt{n}}Z_{\alpha/2},~\hat\theta_n+\frac{\hat\sigma_n}{\sqrt{n}}Z_{\alpha/2}\right].    \label{eq:es_sig}
\end{align}
\end{example}
\section{Simulation studies} \label{sec3}

In this section, we assess the finite-sample performance of the proposed estimation method via extensive simulation studies. We generate the data matrix from the elliptical model 
\[\X_i=\bmu+\xi_i\bSig^{1/2}\U_i, i=1,\ldots, n,\]
where $\bSig=\left(0.5^{|j-k|}\right)$ is a $p\times p$ Toeplitz matrix. According to the examples given in Section \ref{sec2:exm}, we generate $\xi_i^2$ from the following four distributions, all of which satisfy the identifiable assumption $\E\xi_i^2=p$:
\begin{itemize}
	\item[(1)] Chi-squared distribution with $p$ degrees of freedom  $\chi^2_p$ such that $\theta=1$;
        \item[(2)] $\xi_i^2=\omega_i^2/(p+1)$ with $\omega_i\sim $Gamma$(p,1)$; %Gamma distribution with shape parameter $\vartheta=p$ and scale parameter $\beta=1$ such that $\theta=(p+3)/(p+1)$;
        
       % {\color{red} Gamma$(p,1)$ with weight $p+1$$w^2 / {p+1}w^2$ where $w\sim$Gamma$(p,1)$? not quite clear compared with Ex2.2} 
        
	\item[(3)] $F$ distribution with $p$ and 9 degrees of freedom such that $\theta=1.4$;

	\item[(4)] $\xi_i^2=R_{1i}R_{2i}$ with $R_{1i}\sim \exp(1)$ and $R_{2i}\sim \chi^2_p$ such that $\theta=2$.
\end{itemize}

\subsection{Estimation}

For all the numerical studies in this section, we consider various dimensionality levels $p \in \{100, 200, 400, 800, 1600\}$ with the sample size $n=100$. For comparison, we consider an oracle estimator as a benchmark. Note that 
\begin{align*}
    \xi_i^2=(\X_i-\bmu)\trans\bSig^{-1}(\X_i-\bmu),\ i=1,\cdots,n.
\end{align*}
Then, the oracle estimator for $\theta=\E\xi^4/(p(p+2))$ can be constructed as %{\color{red} modify the notation $\hat{\theta}^{Oracle}$ below}
\begin{align*}
    \hat{\theta}^{Oracle}=\frac{1}{np(p+2)}\sum_{i=1}^n  \left[(\X_i-\bmu)\trans\bSig^{-1}(\X_i-\bmu)\right]^2
\end{align*}
when the true parameters $\bmu$ and $\Sig$ are known. Moreover, we implement the coordinate-based method proposed by \cite{ke2018higher}, which is denoted as $\hat{\theta}^{KBF}$;  and the method proposed by \cite{Wang2023boot}, which relies on the moment equation \eqref{eq:mom} and the sample mean and sample covariance matrix, we denote it as $\hat{\theta}^{WL}$. To evaluate the performance of different methods, we summarize the mean and the empirical standard errors over 1000 replications.

Table \ref{tab_err} reports the estimation results of the different methods under the simulation settings described in Section \ref{sec3}. We see that our method performs comparably with  the oracle estimator across all scenarios, indicating that our method is competitive on the kurtosis estimation. Moreover, compared to the methods proposed by \cite{ke2018higher} and \cite{Wang2023boot}, our method has a lower estimation error in most settings. 
 
%{\color{red}check whether the description above is correct.}

\begin{table}[!ht] \centering 
\setlength\tabcolsep{3pt}
	\caption{Estimation results: means and standard errors (in parentheses) for the kurtosis estimates based on 1000 replications.} 
	\label{tab_err} 
	\begin{tabular}{@{\extracolsep{5pt}} lccccc} 
		\\[-1.8ex]\hline 
		\hline \\[-1.8ex] 
		& $p=100$ & $p=200$ & $p=400$ & $p=800$ & $p=1600$ \\ 
		\hline \\[-1.8ex] 
		&\multicolumn{5}{c}{Normal distribution $(\theta=1)$}\\
		$\hat\theta_n$&1.000(0.005) & 1.000(0.002) & 1.000(0.001) & 1.000(0.001) & 1.000(0.000) \\ 
		$\hat{\theta}^{WL}$&0.999(0.005) & 0.999(0.002) & 1.000(0.001) & 1.000(0.001) & 1.000(0.000) \\ 
		$\hat{\theta}^{KBF}$&0.979(0.016) & 0.980(0.011) & 0.980(0.008) & 0.980(0.006) & 0.980(0.004) \\
  $\hat{\theta}^{Oracle}$&1.000(0.003) & 1.000(0.001) & 1.000(0.001) & 1.000(0.000) & 1.000(0.000) \\
  		&\multicolumn{5}{c}{Kotz-type distribution $(\theta=(p+3)/(p+1))$}\\ 
		$\hat\theta_n$&1.020(0.008) & 1.010(0.004) & 1.005(0.002) & 1.002(0.001) & 1.001(0.000) \\ 
		$\hat{\theta}^{WL}$&1.018(0.008) & 1.008(0.004) & 1.004(0.002) & 1.002(0.001) & 1.001(0.000) \\ 
		$\hat{\theta}^{KBF}$&0.997(0.018) & 0.989(0.012) & 0.984(0.008) & 0.982(0.006) & 0.981(0.004) \\
  $\hat{\theta}^{Oracle}$&1.020(0.006) & 1.010(0.003) & 1.005(0.001) & 1.002(0.001) & 1.001(0.000) \\
		&\multicolumn{5}{c}{Multivariate $t$ distribution $(\theta=1.4)$}\\
		$\hat\theta_n$&1.385(0.207) & 1.395(0.293) & 1.388(0.208) & 1.370(0.187) & 1.383(0.184) \\ 
		$\hat{\theta}^{WL}$&1.355(0.186) & 1.365(0.243) & 1.359(0.185) & 1.343(0.168) & 1.355(0.165) \\ 
		$\hat{\theta}^{KBF}$&1.277(0.121) & 1.283(0.135) & 1.280(0.118) & 1.268(0.105) & 1.277(0.107) \\ 
  $\hat{\theta}^{Oracle}$&1.386(0.208) & 1.398(0.309) & 1.387(0.205) & 1.371(0.186) & 1.383(0.185) \\
		&\multicolumn{5}{c}{Multivariate Laplace distribution $(\theta=2)$}\\
		$\hat\theta_n$&2.009(0.225) & 2.000(0.216) & 2.004(0.200) & 2.001(0.212) & 2.011(0.209) \\ 
		$\hat{\theta}^{WL}$&1.919(0.198) & 1.911(0.190) & 1.915(0.176) & 1.912(0.187) & 1.921(0.184) \\ 
		$\hat{\theta}^{KBF}$&1.783(0.153) & 1.776(0.144) & 1.781(0.136) & 1.776(0.141) & 1.784(0.140) \\
  $\hat{\theta}^{Oracle}$&2.006(0.211) & 2.001(0.208) & 2.005(0.196) & 2.001(0.210) & 2.011(0.209) \\
		\hline \\[-1.8ex] 
	\end{tabular} 
\end{table} 
\subsection{Inference}

In this subsection, we assess the performance of the proposed procedure for constructing confidence interval for the kurtosis parameter under the simulation settings described in Section 3. For comparison, we consider the following three types of confidence intervals:
\begin{itemize}
    \item CI$_1$: confidence intervals \eqref{eq:ci1} - \eqref{eq:ci4}, respectively, with the known distribution family;
    \item CI$_2$: confidence interval \eqref{eq:es_sig} without knowing the specific distribution family;
    \item CI$^{KBF}$: confidence interval constructed in \cite{ke2018higher}.
\end{itemize}
We set the significance level $\alpha=0.05$. To measure the reliability and accuracy of different methods for constructing confidence intervals, we calculate the average empirical coverage probability and average width of confidence interval.

 Table \ref{tab:ci} reports the results under different dimensionality levels based on 1000 replications. CI$^{KBF}$ is the most conservative as it produces the widest confidence intervals with slightly inflated coverage
probability. The proposed methods CI$_1$ and CI$_2$ achieve a good balance between reliability (high coverage probability) and accuracy (narrow CI width). Comparing the four distribution families, the proposed methods give more accurate confidence intervals under the normal and Kotz-type distributions, which is in accordance with Theorem \ref{thm2} in Section \ref{sec2-2}.
 
 % We see that the empirical coverage probabilities of our methods(CI$_1$, CI$_2$) are close to the confidence level. For multivariate t distribution, we observe some deviation due to the large variance of $\theta$. Furthermore, the interval widths under normal and Kotz-type distribution are narrow than that under multivariate t and multivariate Laplace distribution, which is consistent with the result in Theorem \ref{thm2}. The interval width of CI$^{KBF}$ is always large such that the coverage probabilities is close to 1 in all cases.
 
\begin{table}[!ht] \centering 
  \caption{Empirical coverage probability and average width of confidence interval of three methods with the significance level $\alpha=0.05$. The results are averaged over 1000 datasets.}
  \label{tab:ci} 
\begin{tabular}{@{\extracolsep{5pt}} cccccccccc} 
\\[-1.8ex]\hline 
\hline \\[-1.8ex] 
$p$&\multicolumn{2}{c}{CI$_1$}&\multicolumn{2}{c}{CI$_2$}&\multicolumn{2}{c}{CI$^{KBF}$}\\
\hline \\[-1.8ex] 
&ECP & AL & ECP & AL & ECP & AL\\
\hline 
&\multicolumn{6}{c}{Normal distribution}\\
\hline
100&$0.934$ & $0.018$ & $0.931$ & $0.018$ & $1$ & $1.194$ \\ 
200&$0.950$ & $0.009$ & $0.942$ & $0.009$ & $0.999$ & $1.189$ \\ 
400&$0.951$ & $0.005$ & $0.949$ & $0.005$ & $0.996$ & $1.182$ \\ 
800&$0.953$ & $0.002$ & $0.926$ & $0.002$ & $0.998$ & $1.180$ \\ 
1600&$0.941$ & $0.001$ & $0.934$ & $0.001$ & $0.995$ & $1.178$ \\ 
\hline
&\multicolumn{6}{c}{Kotz-type distribution}\\
\hline
100&$0.946$ & $0.029$ & $0.928$ & $0.029$ & $0.996$ & $1.261$ \\ 
200&$0.946$ & $0.015$ & $0.933$ & $0.015$ & $0.996$ & $1.220$ \\ 
400&$0.946$ & $0.007$ & $0.925$ & $0.007$ & $0.996$ & $1.199$ \\ 
800&$0.952$ & $0.004$ & $0.939$ & $0.004$ & $0.998$ & $1.189$ \\ 
1600&$0.947$ & $0.002$ & $0.934$ & $0.002$ & $0.995$ & $1.183$ \\
\hline 
&\multicolumn{6}{c}{Multivariate $t$ distribution}\\
\hline
100&$0.900$ & $0.370$ & $0.806$ & $0.374$ & $0.994$ & $2.389$ \\ 
200&$0.908$ & $0.386$ & $0.871$ & $0.373$ & $0.986$ & $2.317$ \\ 
400&$0.925$ & $0.386$ & $0.904$ & $0.393$ & $0.983$ & $2.348$ \\ 
800&$0.924$ & $0.394$ & $0.926$ & $0.403$ & $0.985$ & $2.340$ \\ 
1600&$0.933$ & $0.393$ & $0.933$ & $0.396$ & $0.988$ & $2.327$ \\ 
\hline 
&\multicolumn{6}{c}{Multivariate Laplace distribution}\\
\hline
100&$0.935$ & $0.784$ & $0.934$ & $0.796$ & $0.989$ & $4.092$ \\ 
200&$0.957$ & $0.796$ & $0.954$ & $0.822$ & $0.993$ & $4.094$ \\ 
400&$0.955$ & $0.802$ & $0.961$ & $0.846$ & $0.993$ & $4.064$ \\ 
800&$0.969$ & $0.784$ & $0.959$ & $0.850$ & $0.991$ & $4.054$ \\ 
1600&$0.949$ & $0.795$ & $0.963$ & $0.864$ & $0.992$ & $4.094$ \\ 
\hline \\[-1.8ex]
\end{tabular} 
\end{table} 
\section{Real data analysis}\label{sec4}

In this section, we illustrate the proposed method by analysing the following data sets:
\begin{itemize}
	\item The iris flower dataset, available in the R package $datasets$, consists of 150 samples capturing four measurements from three different species of the iris flower.
	\item The Wisconsin breast cancer dataset, sourced from the UCI Machine Learning Repository, comprises 569 samples featuring 30 distinct characteristics of cell nuclei derived from fine needle aspirates of breast tissue. The samples are labeled as either malignant or benign.
	\item The Parkinson's dataset, available from the UCI Machine Learning Repository, encompasses voice measurements from 188 patients and 64 healthy individuals. With a focus on Parkinson's Disease diagnosis and analysis, the dataset includes 752 features covering a range of voice aspects such as fundamental frequency, spectral residual, and variability.
	\item The prostate cancer data contains 102 samples with 6033 genes, in which 52 are patients with tumor and 50 are normal individuals. The detailed description can be found in \cite{dettling2004bagboosting} and the dataset can be downloaded from \url{http://stat.ethz.ch/~dettling/bagboost.html}.
\end{itemize}

For each dataset, we apply the proposed method to estimate the kurtosis parameter and construct the 95\% confidence interval using \eqref{eq:es_sig} with Cases (i) and (ii).  As shown in Table \ref{tab_realdata}, the malignant breast cancer and the Parkinson's datasets exhibit large kurtosis values, indicating that the data are heavy-tailed. Furthermore, the CI for Case (i) may be incredible for these data because if Case (i) holds, $\theta=1+O(1/p)$.

\begin{table}[!htbp] \centering 
	\setlength\tabcolsep{0pt}
	\caption{Estimation of kurtosis and corresponding 95\% confidence interval for real datasets.} 
	\label{tab_realdata} 
	\begin{tabular}{@{\extracolsep{5pt}} cccccc} 
		\\[-1.8ex]\hline 
		\hline \\[-1.8ex] 
		dataset&class&$(n,p)$&$\theta_n$&\multicolumn{2}{c}{95\%CI}\\
		&&&&Case (i)&Case (ii)\\   
		\hline
		\multirow{3}{*}{iris}&setosa&(50,4)&1.059&(0.563,1.555)&(0.676,1.442)\\
		&versicolor&(50,4)&0.874&(0.465,1.284)&(0.633,1.116)\\
		&virginica&(50,4)&1.008&(0.521,1.495)&(0.655,1.361)\\
		\hline
		\multirow{2}{*}{breast cancer}&malignant&(212,30)& 1.839&(1.317,2.362)&(1.783,1.895)\\
		&benign&(357,30)&1.076&(0.779,1.374)&(1.044,1.108)\\
		\hline
		\multirow{2}{*}{Parkinson}&healthy&(192,752)&5.211&(3.967,6.457)&(4.263,6.161)\\
		&patients&(564,752)&2.271&(1.888,2.652)&(2.200,2.340)\\
		\hline
		\multirow{2}{*}{Prostate}&healthy&(50,6033)&0.825&(0.643,1.008)&(0.696,0.954)\\
		&patients&(52,6033)&1.095&(0.889,1.301)&(0.909,1.280)\\
		\hline
	\end{tabular} 
\end{table}	

\section{Conclusions}
In this paper, we conduct a U-statistic form estimation of the kurtosis parameter in elliptical distributions. Theoretically, we extend the results in existing works\citep{ke2018higher, Wang2023boot}. In special, the consistency of the proposed estimation is established under mild conditions and we derive the asymptotic normality under different moment conditions, which implies different convergence rate of the proposed estimator. A direct generalization of our method is to consider higher moments of $\xi^2$ like \cite{ke2018higher}, and one may use higher order U-statistic estimations. Another extension is to estimate the kurtosis parameter in a multivariate time series with elliptically-distributed noise. We leave all these as future works.

%\section*{Acknowledgement(s)}

\bibliographystyle{abbrvnat}
\bibliography{ref}

\appendix

\section{Proofs and supporting lemmas}\label{sec:a}
We first present several supporting lemmas.
\begin{lemma}\label{lem:1}
	Let $\U \in \mR^{p}$ be a random vector which is uniformly distribution on the unit sphere $\mathcal{S}^{p-1}$. For any $p\times p$ deterministic symmetric matrices $\A$, $\B$ and $\C$, we have
	\begin{align}
		&\E \U\trans\A\U=\frac{1}{p}\tr\A,\label{eq:qua1}\\
		&\E\left(\U\trans\A\U\right)\left(\U\trans\B\U\right)=\frac{\left(\tr\A\tr\B+2\tr\A\B\right)}{p(p+2)},\label{eq:qua2}\\
		&\E\left(\U\trans\A\U\right)\left(\U\trans\B\U\right)\left(\U\trans\C\U\right)\nonumber\\
		&\qquad\qquad\qquad=\frac{\left(\tr\A\tr\B\tr\C+2\tr\A\tr\B\C+2\tr\B\tr\A\C+2\tr\C\tr\A\B+8\tr\A\B\C\right)}{p(p+2)(p+4)},\label{eq:qua3}\\
		&\E\left(\U\trans\A\U\right)^4=\frac{\left(\tr^4\A+12\tr^2\A\tr\A^2+12\tr^2\A^2+32\tr\A\tr\A^3+48\tr\A^4\right)}{p(p+2)(p+4)(p+6)}.\label{eq:qua4}
	\end{align}
\end{lemma}
\begin{proof}[Proof of Lemma \ref{lem:1}]
	\eqref{eq:qua1} and \eqref{eq:qua2} is the direct result of \cite{hu2019high}. By Lemma S4.1 of \cite{hu2019high}, we have
	\begin{align*}
		&\E\left(\U\trans\A\U\right)\left(\U\trans\B\U\right)\left(\U\trans\C\U\right)=\frac{\E(\Z\trans\A\Z)(\Z\trans\B\Z)(\Z\trans\C\Z)}{\E\|\Z\|_2^6},\\
		&\E\left(\U\trans\A\U\right)^4=\frac{\E(\Z\trans\A\Z)^4}{\E\|\Z\|_2^8},
	\end{align*}
	where $\Z\sim N(\bm{0},\bI)$. A detailed result of $\E(\Z\trans\A\Z)(\Z\trans\B\Z)(\Z\trans\C\Z)$ and $\E(\Z\trans\A\Z)^4$ can be found in Theorem 1 of \cite{BAO20101193}, from which we can get \eqref{eq:qua3} and \eqref{eq:qua4}.
\end{proof}
\begin{lemma}\label{lem:2}
	For an centered elliptical distributed variable, 
	\begin{align*}
		\X_0=\xi \bSig^{\frac{1}{2}} \U,
	\end{align*}
	where $\E \xi^2=p$, we have	
	\begin{align*}
		\E \X_0 \trans \X_0=&\tr \bSig;\\
		\var\left(  \X_0 \trans \X_0 \right)=&\E\left( \X_0 \trans \X_0-\tr \bSig\right)^2=2 \theta \tr \bSig^2+(\theta-1)\tr^2\bSig;\\
		\var\left((\X_0\trans\X_0-\tr\bSig)^2\right)=&(\eta_4-4\eta_3-\eta_2^2+8\eta_2-4)\tr^4\bSig+4(3\eta_4-\eta_2^2)\tr^2\bSig^2\nonumber\\
		&+4\left(3\eta_4-6\eta_3-\eta_2^2+4\eta_2\right)\tr^2\bSig\tr\bSig^2+32(\eta_4-\eta_3)\tr\bSig\tr\bSig^3\nonumber\\
     &+48\eta_4\tr\bSig^4;\\
		\var\left((\X_0\trans\X_0-\theta\tr\bSig)^2\right)=&(\eta_4-4\eta_3\eta_2-\eta_2^2+4\eta_2^3)\tr^4\bSig+4(3\eta_4-\eta_2^2)\tr^2\bSig^2\nonumber\\
		&+4\left(3\eta_4-6\eta_3\eta_2+2\eta_2^3+\eta_2^2\right)\tr^2\bSig\tr\bSig^2\nonumber\\
  &+32(\eta_4-\eta_3\eta_2)\tr\bSig\tr\bSig^3+48\eta_4\tr\bSig^4,
	\end{align*}
where $\eta_m=\E\xi^{2m}/\E Y^m$, $m=1,2,\cdots,$ $Y$ follows a chi-squared distribution with $p$ degree of freedom. In particular, $\eta_1=1$ and $\eta_2=\theta$.
\end{lemma}	
\begin{proof}[Proof of Lemma \ref{lem:2}]
	Noting that $\xi$ and $\U \trans \bSig \U$ are independent, we have
	\begin{align*}
		\E \left(\X_0 \trans \X_0  \right)^k=\E \xi^{2k} \cdot \E \left(\U \trans \bSig \U \right)^k,~k=1,2,3,4.
	\end{align*}
	Recalling 
	\begin{align*}
		\eta_k=\frac{\E \xi^{2k}}{p(p+2)\cdots(p+2k-2)},
	\end{align*}
	and $\eta_1=1,~\eta_2=\theta$, we conclude the results from Lemma \ref{lem:1}.
\end{proof}

Both $T_1$, $T_2$ and $T_3$ are invariant with respect to the location transformation, we assume that $\bmu=\bm{0}$ in the rest of the paper.
\subsection{Proof of Theorem \ref{thm1}}
\begin{proof}
	Consider the Hoeffding decomposition of $T_1$, $T_2$ and $T_3$. Defining the kernel functions 
 \begin{align*}
	&k_1(\x_1,\x_2,\x_3,\x_4)=\frac{1}{24}\sum_{(i,j,k,l)=\pi(1,2,3,4)}\left(\|\x_1-\x_2\|_2^2-\|\x_3-\x_4\|_2^2\right)^2,\\
	&k_2(\x_1,\x_2,\x_3,\x_4)=\frac{1}{24}\sum_{(i,j,k,l)=\pi(1,2,3,4)}\|\x_1-\x_2\|_2^2\|\x_3-\x_4\|_2^2,\\
	&k_3(\x_1,\x_2,\x_3,\x_4)=\frac{1}{24}\sum_{(i,j,k,l)=\pi(1,2,3,4)}\left((\x_1-\x_2)\trans(\x_3-\x_4)\right)^2,
\end{align*}
we have
\begin{align*}
	T_1=\frac{1}{4C_n^4}\sum_{i< j<k<l}k_1(\X_i,\X_j,\X_k,\X_l),\\
	T_2=\frac{1}{4C_n^4}\sum_{i< j<k<l}k_2(\X_i,\X_j,\X_k,\X_l),\\
	T_3=\frac{1}{4C_n^4}\sum_{i< j<k<l}k_3(\X_i,\X_j,\X_k,\X_l).
\end{align*}

 Note that
	\begin{align*}
		3k_1(\x_1,\x_2,\x_3,\x_4)=&\left(\|\x_1-\x_2\|_2^2-\|\x_3-\x_4\|_2^2\right)^2+\left(\|\x_1-\x_3\|_2^2-\|\x_2-\x_4\|_2^2\right)^2\\
		&+\left(\|\x_1-\x_4\|_2^2-\|\x_2-\x_3\|_2^2\right)^2,\\
		3k_2(\x_1,\x_2,\x_3,\x_4)=&\|\x_1-\x_2\|_2^2\|\x_3-\x_4\|_2^2+\|\x_1-\x_3\|_2^2\|\x_2-\x_4\|_2^2\\
		&+\|\x_1-\x_4\|_2^2\|\x_2-\x_3\|_2^2,\\
		3k_3(\x_1,\x_2,\x_3,\x_4)=&\left((\x_1-\x_2)\trans(\x_3-\x_4)\right)^2+\left((\x_1-\x_3)\trans(\x_2-\x_4)\right)^2\\
		&+\left((\x_1-\x_4)\trans(\x_2-\x_3)\right)^2.
	\end{align*}
	The conditional expectations of the centered kernel functions is
	\begin{align*}
		k_{11}(\X_1)=&\E\left(k_{1}(\X_1,\X_2,\X_3,\X_4)|\X_1\right)-\E k_{1}(\X_1,\X_2,\X_3,\X_4)\\
		=&\left((\X_1\trans\X_1-\tr\bSig)^2-\E(\X_1\trans\X_1-\tr\bSig)^2\right)+4(\X_1\trans\bSig\X_1-\tr\bSig^2),\\
		k_{21}(\X_1)=&\E\left(k_{2}(\X_1,\X_2,\X_3,\X_4)|\X_1\right)-\E k_{2}(\X_1,\X_2,\X_3,\X_4)\\
		=&2\tr\bSig\left(\X_1\trans\X_1-\tr\bSig\right),\\
		k_{31}(\X_1)=&\E\left(k_{3}(\X_1,\X_2,\X_3,\X_4)|\X_1\right)-\E k_{3}(\X_1,\X_2,\X_3,\X_4)\\
		=&2\left(\X_1\trans\bSig\X_1-\tr\bSig^2\right).
	\end{align*}
	By Lemma \ref{lem:1}, we have
	\begin{align*}
		&\var\left(k_{21}(\X_1)\right)=4(\theta-1)\tr^4\bSig+8\theta\tr^2\bSig\tr\bSig^2\leq12\theta\tr^4\bSig,\\
		&\var\left(k_{31}(\X_1)\right)=4(\theta-1)\tr^2\bSig^2+8\theta\tr\bSig^4\leq12\theta\tr^2\bSig^2=o\left(\tr^4\bSig\right).
	\end{align*}
	For $k_{11}(\X_1)$, note that
	\begin{align*}
		\var\left(\X_1\trans\bSig\X_1\right)=&(\theta-1)\tr^2\bSig^2+2\theta\tr\bSig^4\leq(3\theta-1)\tr^2\bSig^2=o\left(\tr^4\bSig\right),
	\end{align*}
	and by Lemma \ref{lem:2},
	\begin{align*}
		\var\left((\X_1\trans\X_1-\tr\bSig)^2\right)\leq&\E(\X_1\trans\X_1-\tr\bSig)^4=\E\left((\X_1\trans\X_1-\theta\tr\bSig)+(\theta-1)\tr\bSig\right)^4\\
		\leq&4\E(\X_1\trans\X_1-\theta_2\tr\bSig)^4+4\E(\theta-1)^4\tr^4\bSig\\
		\leq& C\tr^4\bSig.	
	\end{align*}
	Next we bound the variance of $k_i(\X_1,\X_2,\X_3,\X_4)$. By Lemma \ref{lem:1},
	\begin{align*}
		\var\left(k_1(\X_1,\X_2,\X_3,\X_4)\right)\leq& \var\left(\left(\|\X_1-\X_2\|_2^2-\|\X_3-\X_4\|_2^2\right)^2\right)\\
		\leq&\E\left(\|\X_1-\X_2\|_2^2-\|\X_3-\X_4\|_2^2\right)^4\\
		\leq&4\E\|\X_1-\X_2\|_2^8\leq64\E(\X_1\trans\X_1)^4=O\left(\tr^4\bSig\right).
	\end{align*}
	Similarly,
	\begin{align*}
		\var\left(k_2(\X_1,\X_2,\X_3,\X_4)\right)\leq&\var\left(\|\X_1-\X_2\|_2^2\|\X_3-\X_4\|_2^2\right)\\
		\leq&\E\|\X_1-\X_2\|_2^4\|\X_3-\X_4\|_2^4\\
		\leq&16\left(\E(\X_1\trans\X_1)^2\right)^2=O\left(\tr^4\bSig\right),\\
		\var\left(k_3(\X_1,\X_2,\X_3,\X_4)\right)\leq&\var\left(\left((\X_1-\X_2)\trans(\X_3-\X_4)\right)^2\right)\\
		\leq&\E\left((\X_1-\X_2)\trans(\X_3-\X_4)\right)^4\leq16\E(\X_1\trans\X_2)^4\\
		=&O(\tr^2\bSig^2)=o(\tr^4\bSig).
	\end{align*}
	Combining the above terms we can get
	\begin{align*}
		&\var\left(\frac{T_1}{\tr^2\bSig}\right)\leq\frac{4^2}{n}\var\left(\frac{k_{11}(\X_1)}{4\tr^2\bSig}\right)+O(\frac{1}{n^2})\leq O(\frac{1}{n}),\\
		&\var\left(\frac{T_2}{\tr^2\bSig}\right)\leq\frac{4^2}{n}\var\left(\frac{k_{21}(\X_1)}{4\tr^2\bSig}\right)+O(\frac{1}{n^2})\leq O(\frac{1}{n}),\\
		&\var\left(\frac{T_3}{\tr^2\bSig}\right)\leq\frac{4^2}{n}\var\left(\frac{k_{31}(\X_1)}{4\tr^2\bSig}\right)+O\left(\frac{\tr^2\bSig^2}{n^2\tr^4\bSig}\right)\leq O\left(\frac{\tr^2\bSig^2}{n\tr^4\bSig}\right)=o(\frac{1}{n}).
	\end{align*}
	Hence we conclude that 
 \begin{align*}
     \frac{T_1}{\tr^2\bSig}-(\theta-1)\topr 0,~\frac{T_2}{\tr^2\bSig}-1\topr 0,~\frac{T_3}{\tr^2\bSig}\topr 0,
 \end{align*}
and then continuous mapping theorem yields $\hat\theta_{n}-\theta\topr 0.$
\end{proof}
\subsection{Proof of Theorem \ref{thm2}}
\begin{proof}
	Note that
	\begin{align*}
		\hat\theta_{n}-\theta=&\frac{T_1+(1-\theta)T_2-2(1+\theta_2)T_3}{T_2+2T_3}.
	\end{align*}
	The kernel function of the U statistic $T_1+(1-\theta)T_2$ is 
	\begin{align*}
		k_1(\X_1,\X_2,\X_3,\X_4)+(1-\theta)k_2(\X_1,\X_2,\X_3,\X_4),
	\end{align*}
	and following the proof of Theorem \ref{thm1}, the conditional expectation is
	\begin{align*}
		k_{11}(\X_1)+(1-\theta)k_{21}(\X_1)=&\left((\X_1\trans\X_1-\theta\tr\bSig)^2-\E(\X_1\trans\X_1-\theta\tr\bSig)^2\right)\\
  &+4(\X_1\trans\bSig\X_1-\tr\bSig^2).
	\end{align*}
	By Lemma \ref{lem:2}, we have
	\begin{align*}
		&\var\left((\X_1\trans\X_1-\theta\tr\bSig)^2\right)=r_{1}\tr^4\bSig+r_2\tr^2\bSig\tr\bSig^2+r_3\tr^2\bSig^2+r_4\tr\bSig\tr\bSig^3+r_5\tr\bSig^4,\\
		&\var\left(\X_1\trans\bSig\X_1\right)=(\theta-1)\tr^2\bSig^2+2\theta\tr\bSig^4.
	\end{align*}
 where 
 	\begin{align*}
		\left\{
		\begin{array}{llllll}
			r_1=\eta_4-4\eta_3\eta_2-\eta_2^2+4\eta_2^3,\\
			r_2=4\left(3\eta_4-6\eta_3\eta_2+2\eta_2^3+\eta_2^2\right),\\
			r_3=4(3\eta_4-\eta_2^2),\\
			r_4=32(\eta_4-\eta_3\eta_2),\\
			r_5=48\eta_4.
		\end{array}
		\right.	
	\end{align*}
By the fact that
\begin{align*}
    &\E\left(\frac{\xi^4}{p^2}\right)=\var\left(\frac{\xi^2}{p}\right)+1,\\
    &\E\left(\frac{\xi^6}{p^3}\right)=\E\left(\frac{\xi^2}{p}-1\right)^3+3\var\left(\frac{\xi^2}{p}\right)+1,\\
    &\E\left(\frac{\xi^8}{p^4}\right)=\E\left(\frac{\xi^2}{p}-1\right)^4+4\left(\frac{\xi^2}{p}-1\right)^3+6\var\left(\frac{\xi^2}{p}\right)+1,
\end{align*}
we have
\begin{align*}
    r_1=&\frac{p^5}{(p+2)^3(p+4)(p+6)}\left[\frac{(p+2)^2}{p^2}\E\left(\frac{\xi^8}{p^4}\right)-4\frac{(p+2)(p+6)}{p^2}\E\left(\frac{\xi^6}{p^3}\right)\E\left(\frac{\xi^4}{p^2}\right)\right.\\
    &\qquad\left.-\frac{(p+2)(p+4)(p+6)}{p^3}\E^2\left(\frac{\xi^4}{p^2}\right)+4\frac{(p+4)(p+6)}{p^2}\E^3\left(\frac{\xi^4}{p^2}\right)\right]\\
    =&\frac{p^5}{(p+2)^3(p+4)(p+6)}\left\{\frac{(p+2)^2}{p^2}\E\left(\frac{\xi^2}{p}-1\right)^4-\frac{8(p^2-4p+12)}{p^3}\var\left(\frac{\xi^2}{p}\right)\right.\\
    &\left.\qquad-\left[\frac{(p+2)(p+4)(p+6)}{p^3}-\frac{24(p+6)}{p^2}\right]\var^2\left(\frac{\xi^2}{p}\right)+\frac{8(p-6)}{p^3}\right.\\
    &\left.\qquad-4\frac{(p+2)(p+6)}{p^2}\E\left(\frac{\xi^2}{p}-1\right)^3\var\left(\frac{\xi^2}{p}\right)+4\frac{(p+4)(p+6)}{p^2}\var^3\left(\frac{\xi^2}{p}\right)\right\}.
\end{align*}
Under the case (i) of Theorem \ref{thm2},
\begin{align*}
    r_1=&\E\left(\frac{\xi^2}{p}-1\right)^4-\var^2\left(\frac{\xi^2}{p}\right)-\frac{8}{p}\var\left(\frac{\xi^2}{p}\right)+\frac{8}{p^2}+o\left(\frac{1}{p^2}\right) \\
    =&\var\left[\left(\frac{\xi^2}{p}-1\right)^2\right]-\frac{8}{p}\left[\var\left(\frac{\xi^2}{p}\right)-\frac{1}{p}\right]+o\left(\frac{1}{p^2}\right)\\
    =&\frac{2(\tau-2)^2}{p^2}+o\left(\frac{1}{p^2}\right).
\end{align*}
Under the case (ii) of Theorem \ref{thm2},
\begin{align*}
  r_1= & \var\left[\left(\frac{\xi^2}{p}-1\right)^2\right]-4\var\left(\frac{\xi^2}{p}\right)\E\left(\frac{\xi^2}{p}-1\right)^3+4\var^3\left(\frac{\xi^2}{p}\right)\\
  =&\var\left[\left(\frac{\xi^2}{p}-1\right)^2-2\var\left(\frac{\xi^2}{p}\right)\cdot\frac{\xi^2}{p}\right].
\end{align*}
With similar arguments, we have:
\begin{itemize}
    \item[] For case (i), 
    \begin{align*}
        &r_2=8(\tau-2)/p+o(1/p),~r_3=8+o(1),\\
        &r_4=64(\tau-2)/p+o(1/p),~r_5=48+o(1),
    \end{align*}
    \item[] For case (ii),
    \begin{align*}
        &r_2=O(1),~r_3=O(1),r_4=O(1),~r_5=O(1).
    \end{align*}
\end{itemize}
Combining all the pieces, we can obtain
\begin{align*}
  \var\left((\X_1\trans\X_1-\theta\tr\bSig)^2\right)=\left\{
  \begin{array}{ll}
      2\left(\frac{\tau-2}{p}\tr^2\bSig+2\tr\bSig^2\right)^2+o\left(\frac{1}{p^2}\right),  &  \rm{case}~(i),\\
      \var\left[\left(\frac{\xi^2}{p}-1\right)^2-2\var\left(\frac{\xi^2}{p}\right)\cdot\frac{\xi^2}{p}\right]+o(1), & \rm{case}~(ii),
  \end{array}
  \right.   
\end{align*}
and 
\begin{align*}
 \var\left(\X_1\trans\bSig\X_1\right)=o\left[\var\left((\X_1\trans\X_1-\theta\tr\bSig)^2\right)\right].   
\end{align*}
 Thus
\begin{align*}
	\var\left(k_{11}(\X_1)+(1-\theta)k_{21}(\X_1)\right)=\sigma_n^2\tr^4\bSig(1+o(1)).
\end{align*}
Similarly, 
\begin{align*}
		\var\left(k_{31}(\X_1)\right)=4\var\left(\X_1\trans\bSig\X_1\right)=o\left(\sigma_n^2\tr^4\bSig\right).
\end{align*}
Next, writing $\Y_1=\X_1-\X_2$ and $\Y_2=\X_3-\X_4$, and noting that
	\begin{align*}
		\E\Y_1\trans\Y_1=&2\tr\bSig,\\
		\E(\Y_1\trans\Y_1)^2=&2(1+\theta)(\tr^2\bSig+2\tr\bSig^2),
	\end{align*}
	thus we have
	\begin{align*}
		&\var\left(k_1(\X_1,\X_2,\X_3,\X_4)+(1-\theta)k_2(\X_1,\X_2,\X_3,\X_4)\right)\\
		\leq&\var\left(\left(\|\X_1-\X_2\|_2^2-\|\X_3-\X_4\|_2^2\right)^2+(1-\theta)\|\X_1-\X_2\|_2^2\|\X_3-\X_4\|_2^2\right)\\
		=&\var\left(\|\Y_1\|_2^4+\|\Y_2\|_2^4-(1+\theta)\|\Y_1\|_2^2\|\Y_2\|_2^2\right)\\
		=&\var\left(\left(\Y_1\trans\Y_1-(1+\theta)\tr\bSig\right)^2+\left(\Y_2\trans\Y_2-(1+\theta)\tr\bSig\right)^2\right.\\
  &\qquad\left.-(1+\theta)(\Y_1\trans\Y_1-2\tr\bSig)(\Y_2\trans\Y_2-2\tr\bSig)\right)\\
		=&2\var\left(\left(\Y_1\trans\Y_1-(1+\theta)\tr\bSig\right)^2\right)+(1+\theta)^2\left(\E(\Y_1\trans\Y_1-2\tr\bSig)^2\right)^2,
	\end{align*}
	\begin{align*}
		&\var\left(\left(\Y_1\trans\Y_1-(1+\theta)\tr\bSig\right)^2\right)\\
		=&\E	\left(\Y_1\trans\Y_1-(1+\theta)\tr\bSig\right)^4-\left(\E\left(\Y_1\trans\Y_1-(1+\theta)\tr\bSig\right)^2\right)^2\\
		=&\E(\Y_1\trans\Y_1)^4-4(1+\theta)\tr\bSig\E(\Y_1\trans\Y_1)^3+6(1+\theta)^2\tr^2\bSig\E(\Y_1\trans\Y_1)^2\\
		&-4(1+\theta)^3\tr^3\bSig\E(\Y_1\trans\Y_1)+(1+\theta)^4\tr^4\bSig-(1+\theta)^2\left((\theta-1)\tr^2\bSig+4\tr\bSig^2\right)^2\\
		=&\E(\Y_1\trans\Y_1)^4-4(1+\theta)\tr\bSig\E(\Y_1\trans\Y_1)^3+4(1+\theta)^2(2\theta+1)\tr^4\bSig\\
		&+16(1+\theta)^2(\theta+2)\tr^2\bSig\tr\bSig^2-16(1+\theta)^2\tr^2\bSig^2.
	\end{align*}
	Expanding $\E(\Y_1\trans\Y_1)^3$ and $\E(\Y_1\trans\Y_1)^4$ we have
	\begin{align*}
		\E(\Y_1\trans\Y_1)^3=&\E\left(\X_1\trans\bSig\X_1+\X_2\trans\bSig\X_2-2\X_1\trans\bSig\X_2\right)^3\\
		=&2\E(\X_1\trans\bSig\X_1)^3+6\E	(\X_1\trans\bSig\X_1)^2\E(\X_1\trans\bSig\X_1)+24\E(\X_1\trans\bSig\X_1)(\X_1\trans\bSig^2\X_1),\\
		\E(\Y_1\trans\Y_1)^4=&\E\left(\X_1\trans\bSig\X_1+\X_2\trans\bSig\X_2-2\X_1\trans\bSig\X_2\right)^4\\
		=&2\E(\X_1\trans\bSig\X_1)^4+8\E	(\X_1\trans\bSig\X_1)^3\E(\X_1\trans\bSig\X_1)+6\left(\E(\X_1\trans\bSig\X_1)^2\right)^2\\
		&+48\E(\X_1\trans\bSig\X_1)^2(\X_1\trans\bSig^2\X_1)+48\E(\X_1\trans\bSig\X_1)(\X_2\trans\bSig\X_2)(\X_1\trans\bSig\X_2)^2\\
		&+16\E(\X_1\trans\bSig\X_2)^4,
	\end{align*}
	applying Lemma \ref{lem:1} we can obtain
	\begin{align*}
		\E(\Y_1\trans\Y_1)^3=&2\left(\eta_3+3\eta_2\right)\tr^3\bSig+12\left(\eta_3+3\eta_2\right)\tr\bSig\tr\bSig^2+16\left(\eta_3+3\eta_2\right)\tr\bSig^3,\\		\E(\Y_1\trans\Y_1)^4=&2(\eta_4+4\eta_3+3\eta_2^2)\tr^4\bSig+24(\eta_4+4\eta_3+3\eta_2^2)\tr^2\bSig\tr\bSig^2\\		&+24\left(\eta_4+4\eta_3+3\eta_2^2\right)\tr^2\bSig^2+64\left(\eta_4+4\eta_3+3\eta_2^2\right)\tr\bSig\tr\bSig^3\\
  &+96\left(\eta_4+4\eta_3+3\eta_2^2\right)\tr\bSig^4,
	\end{align*}
	hence
	\begin{align*}
		&\var\left(\left(\Y_1\trans\Y_1-(1+\theta_2)\tr\bSig\right)^2\right)\\
		=&\left(2\left(\eta_4-4\eta_3\eta_2-\eta_2^2+4\eta_2^3\right)+4(\eta_2-1)^2\right)\tr^4\bSig\\
		&+\left(8\left(3\eta_4-6\eta_3\eta_2+2\eta_2^3+\eta_2^2\right)+16\left(3(\eta_3-\eta_2)+2(\eta_2-1)^2\right)\right)\tr^2\bSig\tr\bSig^2\\
		&+\left(8\left(3\eta_4-\eta_2^2\right)+16\left(6\eta_3+4\eta_2^2-2\eta_2-1\right)\right)\tr^2\bSig^2\\
		&+\left(64\left(\eta_4-\eta_3\eta_2\right)+192\left(\eta_3-\eta_2\right)\right)\tr\bSig\tr\bSig^3\\
		&+96\left(\eta_4+4\eta_3+3\eta_2^2\right)\tr\bSig^4.
	\end{align*}
	Noting
	\begin{align*}
		(1+\theta)^2&\left(\E(\Y_1\trans\Y_1-2\tr\bSig)^2\right)^2\\
		&=4(1+\eta_2)^2\left[(\eta_2-1)^2\tr^4\bSig+4(\eta_2^2-1)\tr^2\bSig\tr\bSig^2+4(\eta_2+1)^2\tr^2\bSig^2\right],
	\end{align*}
	we have
	\begin{align*}
		&\var\left(k_1(\X_1,\X_2,\X_3,\X_4)+(1-\theta)k_2(\X_1,\X_2,\X_3,\X_4)\right)\\
		\leq&2\var\left((\X_1\trans\X_1-\theta\tr\bSig)^2\right)\\
  &+4\left(\tilde{r}_1\tr^4\bSig+\tilde{r}_2\tr^2\bSig\tr\bSig^2+\tilde{r}_3\tr^2\bSig^2+\tilde{r}_4\tr\bSig\tr\bSig^3+\tilde{r}_5\tr\bSig^4\right),
	\end{align*}
 where
\begin{align*}
\left\{
		\begin{array}{llll}
			\tilde{r}_1=&(\eta_2-1)^2\left((\eta_2+1)^2+1\right),\\
			\tilde{r}_2=&4\left(\left(3(\eta_3-\eta_2)+2(\eta_2-1)^2\right)+(\eta_2+1)^2(\eta_2^2-1)\right),\\
			\tilde{r}_3=&4\left(\left(6\eta_3+4\eta_2^2-2\eta_2-1\right)+(\eta_2+1)^4\right),\\
			\tilde{r}_4=&48\left(\eta_3-\eta_2\right),\\
			\tilde{r}_5=&24\left(4\eta_3+3\eta_2^2\right).
		\end{array}
		\right.    
\end{align*}
 With similar arguments of $r_i$, we have:
 \begin{itemize}
    \item[] for case (i), 
    \begin{align*}
     &\tilde{r}_1=  \frac{5(\tau-2)^2}{p^2}+o\left(\frac{1}{p^2}\right),~ \tilde{r}_2=\frac{56(\tau-2)}{p}+ o\left(\frac{1}{p}\right),~\tilde{r}_3=92+o(1),\\
     &\tilde{r}_4=\frac{96(\tau-2)}{p}+ o\left(\frac{1}{p}\right),~\tilde{r}_5=168+o(1),
    \end{align*}
    and then
    \begin{align*}
      &\var\left(k_1(\X_1,\X_2,\X_3,\X_4)+(1-\theta)k_2(\X_1,\X_2,\X_3,\X_4)\right)\\
      \leq&9\left(\frac{\tau-2}{p}\tr^2\bSig+6\tr\bSig^2\right)\left(\frac{\tau-2}{p}\tr^2\bSig+2\tr\bSig^2\right)+\frac{\varepsilon}{p^2}+o\left(\frac{1}{p^2}\right)\\
      =&O\left(\var\left((\X_1\trans\X_1-\theta\tr\bSig)^2\right)\right);
    \end{align*}
    \item[] for case (ii),
    \begin{align*}
        &\tilde{r}_1=\var^2\left(\frac{\xi^2}{p}\right)\left[\left(\var\left(\frac{\xi^2}{p}\right)+2\right)^2+1\right],\\
        &\tilde{r}_2=O(1),~\tilde{r}_3=O(1),\tilde{r}_4=O(1),~\tilde{r}_5=O(1),
    \end{align*}
    and
    \begin{align*}
        &\var\left(k_1(\X_1,\X_2,\X_3,\X_4)+(1-\theta)k_2(\X_1,\X_2,\X_3,\X_4)\right)\\
        =&O\left(\var\left((\X_1\trans\X_1-\theta\tr\bSig)^2\right)\right).
    \end{align*}
\end{itemize}
	Similarly,
	\begin{align*}
		&\var\left(k_3(\X_1,\X_2,\X_3,\X_4)\right)\\
  \leq&\var\left(\left((\X_1-\X_2)\trans(\X_3-\X_4)\right)^2\right)\\
		=&\var\left((\Y_1\trans\Y_2)^2\right)\leq6(1+\theta)\E\left(\Y_2\trans\bSig\Y_2\right)^2\\
		=&12(1+\theta)^2(\tr^2\bSig^2+\tr\bSig^4)\\
  =&O\left(\var\left(k_1(\X_1,\X_2,\X_3,\X_4)+(1-\theta)k_2(\X_1,\X_2,\X_3,\X_4)\right)\right).
	\end{align*}
	Combining all the pieces we can obtain
	\begin{align*}
		\var\left(\frac{T_1+(1-\theta)T_2}{\tr^2\bSig}\right)=&\frac{4^2}{n}\var\left(\frac{k_{11}(\X_1)+(1-\theta)k_{21}(\X_1)}{4\tr^2\bSig}\right)(1+o(1))\\
  =&\frac{\sigma_n^2}{n}(1+o(1)),\\
		\var\left(\frac{T_3}{\tr^2\bSig}\right)=&\frac{1}{n}\var\left(\frac{k_{31}(\X_1)}{\tr^2\bSig}\right)+O\left(\frac{\sigma_n^2}{n^2}\right)=o\left(\frac{\sigma_n^2}{n}\right),
	\end{align*}
	thus we conclude that
	\begin{align*}
		&\frac{\left(T_1+(1-\theta)T_2\right)-\E\left(T_1+(1-\theta)T_2\right)}{\tr^2\bSig}\\
		=&\frac{1}{n}\sum_{i=1}^{n}\frac{k_{11}(\X_i)+(1-\theta)k_{21}(\X_i)-\E\left(k_{11}(\X_i)+(1-\theta)k_{21}(\X_i)\right)}{\tr^2\bSig}+o_p\left(\frac{\sigma_n}{\sqrt{n}}\right),\\
		&\frac{T_3-\E T_3}{\tr^2\bSig}=o_p\left(\frac{\sigma_n}{\sqrt{n}}\right).	
	\end{align*}
	By classical CLT and Slutsky's theorem,
	\begin{align*}
		\frac{\sqrt{n}}{\sigma_n}\left(\frac{T_1+(1-\theta)T_2-2(1+\theta)T_3}{\tr^2\bSig}\right)\tod N(0,1).
	\end{align*} 
	Noting
	\begin{align*}
		\frac{T_2+2T_3}{\tr^2\bSig}\topr 1,
	\end{align*}	
	we finally conclude that
	\begin{align*}
		\frac{\sqrt{n}}{\sigma_n}(\theta_{n}-\theta)=&\frac{\sqrt{n}}{\sigma_n}\left(\frac{T_1+(1-\theta)T_2-2(1+\theta)T_3}{T_2+2T_3}\right)\\
		=&\frac{\sqrt{n}}{\sigma_n}\left(\frac{T_1+(1-\theta)T_2-2(1+\theta)T_3}{\tr^2\bSig}\right)\frac{\tr^2\bSig}{T_2+2T_3}\tod N(0,1).
	\end{align*}
\end{proof}	

\end{document}